\documentclass[a4paper,11pt]{article}

\usepackage[english]{babel}
\usepackage{amsfonts}
\usepackage{latexsym}
\usepackage{amsthm}
\usepackage{amsopn}
\usepackage{amsmath}
\usepackage[all]{xy}
\usepackage{verbatim}
\usepackage{amssymb}
\usepackage{mathrsfs}
\usepackage{float}
\usepackage{eucal}
\usepackage[active]{srcltx}
\usepackage{fullpage}
\usepackage{hyperref}
\usepackage{manyfoot}
\usepackage{amscd}
\usepackage[dvips]{graphics}
\usepackage[dvips]{graphicx}
\usepackage{color}
\usepackage{cite}
\usepackage{pgf,tikz}
\usetikzlibrary{arrows}

\usepackage[a4paper]{anysize}\marginsize{3.5cm}{3.5cm}{1.3cm}{2cm}
\pdfpagewidth=\paperwidth \pdfpageheight=\paperheight

\usepackage[latin1]{inputenc}
\usepackage[all]{xy}
\usepackage{indentfirst}
\usepackage{fancyhdr}
\usepackage{textcomp}
\usepackage{graphicx}
\usepackage{enumerate}

\usepackage{tikz}

\usepackage[bitstream-charter]{mathdesign}

\newtheorem*{theoremA}{Theorem}

\newtheorem{inizio}{Lemma}[section]

\newtheorem{corollary}[inizio]{Corollary}
\newtheorem{proposition}[inizio]{Proposition}
\newtheorem{lemma}[inizio]{Lemma}

\theoremstyle{definition}

\newtheorem{example}[inizio]{Example}
\newtheorem{remark}[inizio]{Remark}

\newcommand{\mQ}{\mathscr{Q}}

\newcommand{\mO}{\mathscr{O}}
\newcommand{\mF}{\mathscr{F}}

\newcommand{\mL}{\mathscr{L}}
\newcommand{\mE}{\mathsf{E}}

\newcommand{\wA}{\widehat{A}}

\newcommand{\fo}{f_* \omega _{S/E}}

\newcommand{\lr}{\longrightarrow}

\renewcommand{\to}{\longrightarrow}

\title{A note on surfaces with $p_g=q=2$ and an irrational fibration}

\author{Matteo Penegini and Francesco Polizzi}

\date{}

\begin{document}

\maketitle

\begin{abstract}
We study several examples of surfaces with $p_g=q=2$ and maximal Albanese dimension that are endowed with an irrational fibration. 
\end{abstract}

\noindent








\Footnotetext{{}}{\textit{2010 Mathematics Subject
Classification}: 14J29}

\Footnotetext{{}} {\textit{Keywords}: {Surfaces of general type,
fibrations, Albanese map}}



\section{Introduction}\label{sec.intro}

The classification of surfaces $S$ of general type  with
$\chi(\mathscr{O}_S)=1$, i.e. $p_g(S)=q(S)$, is currently an active area of research, see for instance \cite{BCP06}. In this case some well-known results imply $p_g \leq 4$ . While
surfaces with $p_g=q=4$ and $p_g=q=3$ have been completely
described in \cite{D82}, \cite{CCML98}, \cite{HP02}, \cite{Pi02},
the classification of those with $p_g=q=2$ is still missing, albeit some new interesting examples were recently discovered (\cite{CH06,Pe11,PP14,PP13a,PP13b}). For a recent survey on this topic we refer the reader to  \cite{Pe12}.

One of the most useful techniques used to understand the geometry of an algebraic surface $S$ is the study of its fibrations $f\colon S \longrightarrow C$, where $C$ is a smooth curve. If $g(C)\geq 1$ we say that the $f$ is \emph{irrational}, and if all its smooth fibres are isomorphic we say that $f$ is \emph{isotrivial}. The study of irrational fibrations on surfaces with $p_g=q=2$ was started by Zucconi in \cite{Z03}. Later on, however, it was found that Zucconi's results were incomplete: see \cite{Pe11}, where the first author deals with the isotrivial case.

If the image of the Albanese map of $S$ is a curve then everything is known: $S$ is a so-called generalized hyperelliptic surface, i.e. a quotient $S= (C_1 \times C_2)/G$ by the diagonal action of a finite group  $G$ on $C_1 \times C_2$ such that the Galois morphism $C_1 \to C_1/G$ is unramified and $C_2 /G \cong \mathbb{P}^1$, see \cite{cat00, Z03b, Pe11}. Therefore it only remains to investigate the case where $S$ has \emph{maximal Albanese dimension}, i.e. its Albanese map $\alpha \colon S \to \textrm{Alb}(S)$ is generically finite; in this situation the base of any irrational fibration on $S$ is an elliptic curve $E$ (see Proposition \ref{prop:irr-penc}).  An important invariant of the fibration $f$ is  the push-forward of the relative canonical bundle $f_* \omega_{S/E}$  (see for instance \cite[Chapter III]{BHPV03}), and the aim of this paper is to explicitly calculate this invariant in many different examples related to our previous work on the subject, see \cite{PP13a, PP13b, PP14}.

This article is organized as follows. In Section \ref{sec.prel} we fix the notation and  the terminology and we state some technical results needed in the sequel of the paper. Moreover, by using methods borrowed from \cite{Fu78,Fu78b,Ba00,CD13}, we deduce the following structure result, see Propositions \ref{thm:main} and \ref{prop:car r=1}.
\begin{theoremA} \label{thm:main0}
Let $S$ be a minimal surface of general type with $p_g=q=2$ and maximal Albanese dimension. Let $f \colon S
\to E$ be an irrational fibration whose general fibre $F$ has genus $g$. Then $g(E)=1$ and there exist $p \in E$ and $r \geq 1$ such that
\begin{equation*}
\fo = \mO_E \oplus \mE_p(r, \, 1) \oplus \bigoplus_{i=2}
^{g-r} \mathscr{Q}_i.
\end{equation*}
Here $\mE_p(r, \, 1)$ denotes the unique indecomposable vector bundle on $E$ of rank $r$ and determinant $\mO_E(p)$, whereas the $\mathscr{Q}_i$ are pairwise non-isomorphic, non-trivial line bundles in $\mathrm{Tors}(\mathrm{Pic}^0 E)$.

Finally, $r=1$ if and only if the divisor $K_S - F_p$ is effective.
\end{theoremA}

This theorem corrects and extends Zucconi's results quoted above. For instance, in \cite{Z03} only the case where $r=1$ is considered, and the existence of non-isotrivial irrational  fibrations is overlooked. See Remark \ref{rem_Zucconi} for more details. 

In Section \ref{sec:examples} we provide several examples with $r =1$ and $r \geq 2$, both in the isotrivial case (Examples \ref{ex:pen-1}, \ref{ex:pen-4}, \ref{ex:pen-5}, \ref{ex:pen-6}) and in the non-isotrivial one (Examples \ref{ex:K^2=6-d=2}, \ref{ex:K^2=5,6}, \ref{ex:diagonal}, \ref{ex:r non limited}). In particular, we show the existence of surfaces $S$ with $p_g=q=2$ and $(K_S^2, \, \deg \alpha) \in \{(4, \, 2)\}, \{(5,\, 3), \, (6, \, 2), \, (6, \, 4) \}$, such that $S$ contains an \emph{infinite} family $f_n \colon S \to E$ of non-isotrivial,  irrational fibrations whose fibre genera form an unbounded sequence. In the case $(K_S^2, \, \deg \alpha)= (4, \, 2)$ the integer $r(f_n)$ can be arbitrarily large, too: in fact, $r(f_n)=n^2+1$. 

In the discussion of the non-isotrivial case we need some explicit computations on $(1, \, 2)$-polarized abelian surfaces. For the reader's convenience, we put them in the Appendix.

\bigskip
\textbf{Acknowledgments} It is a pleasure to thank I. Bauer, S. Rollenske and F. Zucconi for stimulating discussions and useful suggestions, and the anonymous referee whose comments helped to improve the presentation of these results.

Both authors were partially supported by
Progetto MIUR di Rilevante Interesse Nazionale \emph{Geometria
delle Variet$\grave{a}$ Algebriche e loro Spazi di Moduli} and by
GNSAGA - INdAM. They are also grateful to the Hausdorff Research Institute for Mathematics (Bonn)
for the invitation and hospitality during the Junior Trimester Program \emph{Algebraic Geometry}
(February-April, 2014).

\bigskip
\textbf{Notation and conventions.} We work over the field
$\mathbb{C}$ of complex numbers. Throughout the paper we use
italic letters for line bundles  and
 capital letters for the corresponding Cartier divisors, so we write for
instance $\mathcal{L}=\mathcal{O}_S(L)$.

By \emph{surface} we mean a projective, non-singular surface
$S$, and for such a surface $\omega_S=\mO_S(K_S)$ denotes the
canonical class, $p_g(S)=h^0(S, \, \omega_S)$ is the
\emph{geometric genus}, $q(S)=h^1(S, \, \omega_S)$ is the
\emph{irregularity} and $\chi(\mathscr{O}_S)=1-q(S)+p_g(S)$ is the
\emph{Euler-Poincar\'e characteristic}.

If $A$ is an abelian variety and $\wA:= \textrm{Pic}^0(A)$ is its
dual variety,  $A[2]$ and $\widehat{A}[2]$ stand for
the corresponding subgroups of $2$-division points. If $x \in A$
, we write $t_x \colon A \to A$ for the translation by $x$. Given any line bundle $\mathcal{L}$
on $A$
 we denote by $\phi_{\mathcal{L}}$ the morphism $\phi_{\mL} \colon A \to
\wA$ given by $x \mapsto t^*_x \mL \otimes \mL^{-1}$. If
$c_1(\mL)$ is non-degenerate then $\phi_{\mL}$ is an isogeny, and
we write $K(\mL)$ for its kernel.

\section{Preliminaries} \label{sec.prel}
Given any fibration $f \colon S \to C$ with general fibre $F$, we define the \emph{slope} of $f$ as the ratio
\begin{equation*}
\lambda(f)= K^2_{S/C} / \Delta(f),
\end{equation*}
where
\begin{equation*}
\begin{split}
K^2_{S/C} & = K_S^2 - 8(g(C)-1)(g(F)-1), \\ \Delta(f)&= \chi(\mO_S)- (g(C)-1)(g(F)-1)
\end{split}
\end{equation*}
and $g(C)$ and $g(F)$ are the genera of $C$ and $F$, respectively.

\begin{proposition} \label{prop:xiao}
Let $f \colon S \to C$ be a relatively minimal fibration with
$\lambda(f)=4$ and $q(S) > g(C)$. Then necessarily $q(S)=g(C)+1$
and moreover  $f_* \omega_{S/C}$ is the direct sum of $\mO_C$ and
a semistable sheaf of rank $g(F)-1$.
\end{proposition}
\begin{proof}
See \cite[Theorem 3 p. 462]{Xi87}.
\end{proof}

Recall that a fibration $f\colon S \to C$ is said to be \emph{isotrivial} if all its smooth fibres are isomorphic.
\begin{proposition} \label{thm:numerical}
Let $f \colon S \lr C$ be a relatively minimal isotrivial fibration,
with $S$ non ruled and $g(C) \geq 1$. If
 $S$ is not isogenous to an unmixed product $($see Subsection \ref{subsec:iso} for the definition$)$ we have the sharp inequality
\begin{equation*} 
K_S^2 \leq 8 \chi(\mO_S)-2
\end{equation*}
and if equality holds then $S$ is a minimal surface of general
type whose canonical model
has precisely two ordinary double points as singularities.
Moreover, under the further assumption that $K_S$ is ample, we
have the sharp inequality
\begin{equation*} 
K_S^2 \leq 8 \chi(\mO_S)-5.
\end{equation*}
In particular, if $K_S$ is ample and
\begin{equation*}
8 \chi(\mO_S)-5 < K_S^2 < 8  \chi(\mO_S)
\end{equation*}
then $f$ is not isotrivial.
\end{proposition}
\begin{proof}
See \cite{Pol10}.
\end{proof}

In the sequel, $S$ will be a smooth, minimal surface with
$p_g=q=2$. We denote by
$A:=\textrm{Alb}(S)$ the Albanese variety of $S$ and by $\alpha
\colon S \to A$ the corresponding Albanese map. We also assume that $S$ has
maximal Albanese dimension, i.e. that $\alpha$ is generically finite onto the abelian surface $A$.
\begin{proposition} \label{prop:irr-penc}
If $f \colon S \to E$ is an irrational fibration, then $E$ is an
elliptic curve.
\end{proposition}
\begin{proof}
We must have $1 \leq g(E) \leq 2$, because $q(S)=2$ and the fibration
is irrational. If $g(E)=2$, using the embedding $E \hookrightarrow \textrm{Jac}(E)$ and the universal property of the
Albanese map, we obtain a morphism $\beta \colon A \to \textrm{Jac}(E)$, whose image is isomorphic to the curve $E$. On the other hand, the image of $\beta$ must be a translated of an abelian subvariety of $A$, hence $g(E)=1$, contradiction.
\end{proof}

\begin{remark} \label{prop:A-not-simple}
{\rm If $S$ admits an irrational fibration $f \colon S \to E$, then $A$ is a non-simple
abelian surface. In fact, since $E$ is an elliptic
curve (Proposition \ref{prop:irr-penc}), the universal property of the Albanese map yields a
surjective morphism $A \to E$, whose kernel is a
$1$-dimensional subtorus of $A$.}
\end{remark}

If $E$ is an elliptic
curve and $p \in E$, we set $\mE_p(1, \,1):= \mO_E(p)$ and for
all $r \geq 2$ we denote by $\mE_{p}(r, \, 1)$ the unique
indecomposable vector bundle rank $r$ on $E$ defined
recursively by the short exact sequence
\begin{equation*}
0 \lr \mO_E \lr \mE_{p}(r, \, 1) \lr \mE_{p}(r-1, \, 1) \lr 0.
\end{equation*}
For any $\mQ \in \textrm{Pic}^0\,E$ we have $h^0(E, \, \mE_{p}(r, \, 1)\otimes \mQ)=1$ and $h^1(E, \, \mE_{p}(r, \, 1)\otimes \mQ)=0$, see \cite[Lemma 8 and 15]{At57}.

\begin{lemma} \label{lemma:indec}
Let $\mathscr{A}$ be an indecomposable vector bundle over an elliptic curve $E$.
If $\deg \mathscr{A}= d <0$, then $H^0(E, \, \mathscr{A})=0$.
\end{lemma}
\begin{proof}
We work by induction on $r:= \textrm{rank}\, \mathscr{A}$. If $r=1$ the result is clear. Assume now that the claim is true for any vector bundle of rank $r-1$. If $h^0(E, \, \mathscr{A})>0$, then there exists an exact sequence 
\begin{equation*}
0 \to \mathscr{O}_E \to \mathscr{A} \to \mathscr{B} \to 0,
\end{equation*} 
where $\mathscr{B}$ is a vector bundle on $E$ of degree $d$ and rank $r-1$. Hence, by Serre duality and the inductive hypothesis, we infer
\begin{equation*}
\textrm{Ext}^1(\mathscr{B}, \, \mathscr{O}_E)= H^1(E, \, \mathscr{B}^*)=H^0(E, \, \mathscr{B})=0, 
\end{equation*}
a contradiction because $\mathscr{A}$ is indecomposable.
\end{proof}

\begin{proposition} \label{thm:main}
Let $f \colon S
\to E$ be an irrational fibration on a minimal surface with $p_g=q=2$ and maximal  Albanese dimension. Then
\begin{equation*}
\fo = \mO_E \oplus \mE_p(r, \, 1) \oplus \bigoplus_{i=2}
^{g-r} \mathscr{Q}_i,
\end{equation*}
where $p \in E$, $r \geq 1$ and the $\mathscr{Q}_i$ are pairwise non-isomorphic, non-trivial torsion line bundles on $E$.
\end{proposition}

\begin{proof}
By \cite[Theorem 2.7 and Theorem 3.1]{Fu78} we can write 
\begin{equation*}
f_* \omega_{S/E} = \mO_E^{\oplus h} \oplus \mF, 
\end{equation*}
where $h=h^1(E, \, f_* \omega_S)$ and $\mF$ is a
locally free sheaf such that $H^1(E, \, \mF \otimes \omega_E)=0$. 
 
From \cite[Proposition 1.8 (i)]{Ba00} it follows $h=q(S)-g(E)=1$. Moreover, by \cite{Fu78b} and \cite[Corollary 21]{CD13} we have
\begin{equation} \mF = \mathscr{A} \oplus  \bigoplus \mQ_i,
\end{equation}
where $\mathscr{A}$ is an ample vector bundle on $E$ and and each $\mQ_i$ is a non-trivial, torsion line bundle. 

Since $\mathscr{A}$ is ample, each indecomposable direct summand $\mathscr{A}_i$ of $\mathscr{A}$ has degree $>0$, hence $h^0(E, \, \mathscr{A}_i)>0$ (\cite[Lemma 1.1 and Proposition 1.2]{H71}). But $h^0(E, \, \mathscr{A})=1$, so $\mathscr{A}$ is indecomposable and by \cite[p. 434]{At57} we infer that there exists $p \in E$ such that $\mathscr{A}=E_p(r, \, 1)$, where $r=\textrm{rank}\, \mathscr{A}$.  

It remains to show that the line bundles $\mQ_i$ are pairwise non-isomorphic. By contradiction, and without loss of generality, assume $\mQ_1= \mQ_2$. Then
\begin{equation*}
f_* \omega_S  \otimes \mathscr{Q}_1^{-1}= \mQ_1^{-1} \oplus (\mE_p(r,  \, 1) \otimes \mQ_1^{-1}) \oplus \mO_E \oplus \mO_E \oplus \bigoplus_{i=3}^{g-r} (\mQ_i \otimes \mQ_1^{-1}),
\end{equation*}
which implies, by projection formula, $h^0(S, \, \omega_S \otimes f^* \mQ_1^{-1}) \geq 3$. Now Serre duality and projection formula yield 
\begin{equation*}
h^2(S, \, \omega_S \otimes f^* \mQ_1^{-1})= h^0(S, \, f^* \mQ_1)=h^0(E, \, \mQ_1)=0,
\end{equation*}
hence $h^1(S, \, \omega_S \otimes f^* \mQ_1^{-1}) \geq 2$.
On the other hand, a direct calculation using the Leray spectral sequence of $f \colon S \to E$ shows that, given $\mQ \in \textrm{Pic}^0 E$, we have
\begin{equation} \label{eq:jump}
h^1(S, \, \omega_S \otimes f^* \mQ)=\left\{%
\begin{array}{ll}
    2 & \hbox{if } \; \mQ = \mO_E \\
    1 & \hbox{if } \; \mQ = \mQ_i^{-1} \hbox{ for some }i \\
    0 & \hbox{otherwise.} \\
\end{array}%
\right.
\end{equation}
This is a contradiction.
\end{proof}


We now denote the integer $r$ which appears in Proposition \ref{thm:main} by
$r(f)$. If $p \in E$, we write $F_p$ for the
fibre of $f \colon S \to E$ over $p$, namely $F_p = f^{-1}(p)$. The following result shows that $r(f)$ is related to the existence in paracanonical systems of reducible curves  containing fibres of $f$ (we refer the reader to \cite{Be88} and \cite{MPP13} for generalities and results about paracanonical systems on surfaces).

\begin{proposition} \label{prop:car r=1}
Let $f \colon S \to E$ be an irrational fibration on a minimal
surface with $p_g=q=2$ and maximal Albanese dimension.  Then the following are equivalent:
\begin{itemize}
\item[$\boldsymbol{(1)}$] $r(f)=1;$
\item[$\boldsymbol{(2)}$] for any $\eta \in \mathrm{Pic}^0 E$ there
exists a $(unique)$ point $p_{\eta} \in E$ such that the linear
system $|K_S + f^*\eta - F_{p_{\eta}}|$ is not empty. More precisely,
$h^0(S, \, K_S + f^*\eta - F_{p_{\eta}})=1;$
\item[$\boldsymbol{(3)}$] there exists a $(unique)$ point $p \in E$ such that the
linear system $|K_S - F_p|$ is not empty. More precisely, $h^0(S, \,
K_S - F_p)=1.$
\end{itemize}
\end{proposition}
\begin{proof}
$\boldsymbol{(1)} \Rightarrow \boldsymbol{(2)}$ If $r(f)=1$,
 there exists $p \in E$ such that $f_* \omega_S = \mO_E \oplus
\mO_E(p) \oplus \bigoplus_{i=2} ^{g-1} \mathscr{Q}_i$, so
projection formula yields, for any $q \in E$,
\begin{equation*}
\begin{split}
H^0(S, \, K_S + f^*\eta - F_q)& =H^0(E, \, f_*(\omega_S \otimes f^*
\eta \otimes f^* \mO_E(-q)))\\ & = H^0(E, \, \mO_E(-q) \otimes \eta)
\oplus H^0 (E, \, \mO_E(p-q) \otimes \eta) \oplus \bigoplus_{i=2}
^{g-1} H^0(E, \,\mathscr{Q}_i \otimes \mO_E(-q) \otimes \eta) \\
& = H^0 (E, \, \mO_E(p-q) \otimes \eta).
\end{split}
\end{equation*}
Now it suffices to choose as $p_{\eta}$ the unique point $q$ such
that $\mO_E(q) = \mO_E(p) \otimes \eta$. \\ \\
$\boldsymbol{(2)} \Rightarrow \boldsymbol{(3)}$ Clear. \\ \\
$\boldsymbol{(3)} \Rightarrow \boldsymbol{(1)}$ Assume that there
exists $p \in E$ such that $|K_S - F_p|$ is not empty. Using Proposition
\ref{thm:main} and projection formula we can write
\begin{equation*}
\begin{split}
H^0(S, \, K_S - F_p)& =H^0(E, \, f_*(\omega_S  \otimes f^*
\mO_E(-p)))\\ & = H^0(E, \, \mO_E(-p)) \oplus H^0 (E, \,
\mE_p(r, \,1) \otimes \mO_E(-p)) \oplus \bigoplus_{i=2} ^{g-r}
H^0(E, \, \mathscr{Q}_i \otimes \mO_E(-p) ) \\ & = H^0 (E, \,
\mE_p(r, \,1) \otimes \mO_E(-p)).
\end{split}
\end{equation*}
The indecomposable vector bundle $\mE_p(r, \,1) \otimes \mO_E(-p)$ has degree $1-r$, which is a negative integer for $r >1$. Therefore Lemma \ref{lemma:indec} implies that 
it has a global section if and only if $r=1$, i.e. when $\mE_p(r, \,1)=\mO_E(p)$, and this completes the proof.
\end{proof}

\begin{corollary} \label{prop:genus fibre}
Let $f \colon S \to E$ be an irrational fibration on a minimal
surface with $p_g=q=2$ and maximal Albanese dimension. If $r(f)=1$, then $2 \leq g(F) \leq 5$.
\end{corollary}
\begin{proof}
If $r(f)=1$ then $K_S - F$ is effective (Proposition \ref{prop:car r=1}).
Since $K_S$ is nef it follows $K_S(K_S - F) \geq 0$,
that is $K_SF \leq K_S^2$. By Bogomolov-Miyaoka-Yau inequality (\cite[Chapter VII]{BHPV03}) we have $K^2_S \leq 9 \chi(\mO_S)= 9$, hence $g(F) \leq 5$.
\end{proof}

\begin{remark}\label{rem_Zucconi}In \cite[Corollary 2.4 and Theorem 2.8]{Z03} it is stated that all irrational fibrations on surfaces with $p_g=q=2$ and maximal Albanese dimension satisfy $r(f)=1$ and $2 \leq g(F) \leq 5$, and that they are isotrivial as soon as $g(F) >2$. We will see in the next section  that this is not true: for example, we will show the existence of non-isotrivial, irrational fibrations with $r(f)\geq 2$ and  $g(F)$ arbitrarily large, see Examples \ref{ex:K^2=6-d=2}, \ref{ex:K^2=5,6} and \ref{ex:diagonal}. We found that some of the arguments in \cite{Z03} are actually incomplete: for instance, the analysis of the case $\mL=\mO_E$ is missing in \cite[proof of Lemma 2.3]{Z03}.
\end{remark}


\section{Examples} \label{sec:examples}

\subsection{Isotrivial examples} \label{subsec:iso}
Isotrivial, irrational fibrations on surfaces with $p_g=q=2$ were
 classified in \cite{Pe11}. The aim of this section is to compute the integer $r(f)$  for some of them. Before doing this we introduce some notation and terminology, referring the reader to \cite{Pol10} for further details.

A smooth surface $S$ is called a \emph{standard isotrivial
fibration} if there exists a finite group $G$, acting faithfully on
two smooth projective curves $C_1$ and $C_2$ and diagonally on their
product, so that $S$ is isomorphic to the minimal desingularization
of $T:=(C_1 \times C_2)/G$. We denote such a desingularization by
$\lambda \colon S \to T$. In particular, if the action of $G$ is free then  $T$ is smooth and we call $S=T$ a
surface \emph{isogenous to an unmixed product}.

If $\lambda \colon S \to T=(C_1 \times C_2)/G$ is any standard
isotrivial fibration, composing the two morphism $T
\lr C_1/G$ and $T \lr C_2/G$ with $\lambda$ one obtains
two fibrations $f_1 \colon S \lr C_1/G$ and $f_2 \colon S
\lr C_2/G$, whose smooth fibres $F_1$ and $F_2$ are isomorphic to $C_2$ and $C_1$, respectively. Moreover $F_1F_2 =|G|$.

We denote by $\boldsymbol{m_1}$ and $\boldsymbol{m_2}$ the vectors of branching data of
$f_1$ and $f_2$, respectively. For istance $\boldsymbol{m_1}=(2, \, 2)$ means that $f_1$ has two branching points, both with branching order $2$. The irregularity of $S$ is $q(S)=g(C_1/G) +g(C_2/G)$, see
\cite{Fre71}. Moreover, if $g(C_1)$ and $g(C_2)$ are both strictly positive, the surface $S$ is necessarily a minimal model.

\begin{lemma} \label{lemma: abelian}
Let $S=(C_1 \times C_2)/G$ be a surface isogenous to an unmixed product, and assume that $G$ is an abelian group.
Then $f_{i*} \omega_S$ splits into a direct sum of line bundles.
\end{lemma}
\begin{proof}
Roughly speaking, this follows from the fact that the irreducible representations of a finite, abelian group are $1$-dimensional, together with the structure results for abelian covers given in \cite{Pa91}. Indeed, let us consider the commutative diagram
\begin{equation} \label{diag:abelian}
\begin{CD}
C_1 \times C_2  @> \Phi>> S\\
@V{\pi_i}VV  @VV{f_i}V\\
C_i @> \phi_i>> C_i/G.\\
\end{CD}
\end{equation}
Since $G$ is abelian, there exists a direct sum decomposition
\begin{equation} \label{eq:characters-ab}
\Phi_* \omega_{C_1 \times C_2} = \omega_S \oplus \bigoplus_{\chi \in \widehat{G}^*} (\omega_S \otimes \mL_{\chi}),
\end{equation}
where $\widehat{G}^ *$ is the subset of non-trivial characters of $G$ and the $\mathscr{L}_{\chi}$ are line bundles on $S$. Applying $f_{i*}$ to both sides of \eqref{eq:characters-ab}
it follows that $f_{i*} \omega_S$ is a direct summand of $f_{i*} \Phi_* \omega_{C_1 \times C_2}$. Using the commutativity of \eqref{diag:abelian}, we obtain
\begin{equation*}
\begin{split}
f_{i*} \Phi_* \omega_{C_1 \times C_2} & = \phi_{i*} \pi_{i*} \omega_{C_1 \times C_2}\\
& = (\phi_{i*} \omega_{C_i})^{\oplus g(C_{i+1})} = \omega_{C_i/G}^{\oplus g(C_{i+1})} \oplus \bigoplus_{\chi \in \widehat{G}^*} (\omega_{C_i/G} \otimes \mathscr{M}_{\chi})^{\oplus g(C_{i+1})},
\end{split}
\end{equation*}
where the integer $i+1$ has to be intended $\textrm{mod}\, 2$ and the $\mathscr{M}_{\chi}$ are line bundles on $C_i/G$.
Therefore we are done,  because the right-hand side is a direct sum of line bundles and the decomposition of a vector bundle into indecomposable ones is unique up to reordering of the summands (\cite{At56}).
\end{proof}
We are now ready to compute $r(f_i)$ in some of isotrivial  examples. We only give the construction data and
we refer the reader to \cite[Sections 3 and 4]{Pe11} for the details. We write $E_i$ for the elliptic curve $C_i/G$.

\begin{example} \label{ex:pen-1}  $g(C_1)=3, \quad g(C_2)=3, \quad
G=\mathbb{Z}/2 \mathbb{Z} \times \mathbb{Z}/2 \mathbb{Z}, \quad
K_S^2=8$. \\
$\boldsymbol{m_1}=(2,\,2), \quad \boldsymbol{m_2}=(2, \, 2)$.

Since in this case $G$ is abelian, Lemma \ref{lemma: abelian} yields
\begin{equation*}
r(f_1)=1, \quad  r(f_2)=1.
\end{equation*}
\end{example}

\begin{example} \label{ex:pen-4}
$g(C_1)=2, \quad g(C_2)=2, \quad G=\mathbb{Z}/2 \mathbb{Z}, \quad K_S^2=4$. \\
$\boldsymbol{m_1}=(2, \, 2), \quad \boldsymbol{m_2}=(2, \,2)$.

In this case
\begin{equation*}
r(f_1)=1, \quad r(f_2)=1,
\end{equation*}
because $g(C_i)=2$.
\end{example}

\begin{example} \label{ex:pen-5}
$g(C_1)=3, \quad g(C_2)=3, \quad G=Q_8$ or $G=D_8, \quad K_S^2=4$. \\
$\boldsymbol{m_1}=(2), \quad \boldsymbol{m_2}=(2)$.

In this case
\begin{equation*}
r(f_1)=2, \quad  r(f_2)=2.
\end{equation*}
Indeed, $\lambda(f_i)=4$
so by Proposition \ref{prop:xiao} it follows that $f_{i*}
\omega_{S/E_i}$ is the direct sum of $\mathscr{O}_{E_i}$ with a
semistable vector bundle of rank $2$ and degree $1$; such a bundle
is necessarily of the form $\mE_p(2, \, 1)$ for some $p \in E$, see
Proposition \ref{thm:main}.
\end{example}

\begin{example} \label{ex:pen-6}
$g(C_1)=3, \quad g(C_2)=3, \quad G=S_3, \quad K_S^2=5$.  \\
$\boldsymbol{m_1}=(3), \quad \boldsymbol{m_2}=(3)$.

In this case
\begin{equation*}
r(f_1)=2, \quad  r(f_2)=2.
\end{equation*}
In fact, we have
\begin{equation*}
\textrm{Sing}(T) = \frac{1}{3}(1,\, 1)+ \frac{1}{3}(1, \, 2),
\end{equation*}
so $S$ contains a $(-3)$-curve $W$ and two $(-2)$-curves $Z_1$, $Z_2$ such that
$WZ_1 =WZ_2 =0$ and $Z_1 Z_2 =1$. Using the results in \cite[Section 2]{Pol10} one checks that the linear equivalence classes of the fibres $F_i$ of $f_i \colon S \to E_i$ are
\begin{equation*}
F_1 = 3Y_1 + W + 2Z_1 + Z_2, \quad F_2 = 3Y_2 + W + Z_1+ 2 Z_2,
\end{equation*}
where the $Y_i$ satisfy $Y_i^2=-1$ and $K_S Y_i=1$.  Moreover $F_1F_2 = |G|=6$,  hence we infer
$Y_1 Y_2 =0$, see Figure \ref{fig:M1}.
\definecolor{ffqqtt}{rgb}{1,0,0.2}
\definecolor{qqqqff}{rgb}{0,0,1}

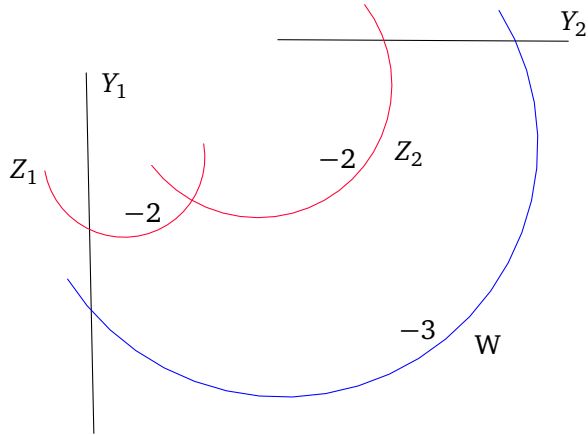
\begin{figure}[H]
\centering
\begin{tikzpicture}[cap=round,line join=round,x=1.0cm,y=1.0cm]
\foreach \x in {1,2,3,4,5,6,7,8}
\foreach \y in {1,2,3,4,5,6}
\clip(0,0) rectangle (8.44,6.3);
\draw (1.34,5.1)-- (1.44,0.32);
\draw (3.86,5.54)-- (7.68,5.52);
\draw [shift={(1.84,3.98)},color=ffqqtt]  plot[domain=-2.97:0.17,variable=\t]({1*1.06*cos(\t r)+0*1.06*sin(\t r)},{0*1.06*cos(\t r)+1*1.06*sin(\t r)});
\draw [shift={(3.6,4.94)},color=ffqqtt]  plot[domain=-2.49:0.65,variable=\t]({1*1.76*cos(\t r)+0*1.76*sin(\t r)},{0*1.76*cos(\t r)+1*1.76*sin(\t r)});
\draw [shift={(3.93,4.15)},color=qqqqff]  plot[domain=-2.58:0.56,variable=\t]({1*3.35*cos(\t r)+0*3.35*sin(\t r)},{0*3.35*cos(\t r)+1*3.35*sin(\t r)});
\draw (1.40,5.22) node[anchor=north west] {$Y_1$};
\draw (7.44,6.05) node[anchor=north west] {$Y_2$};
\draw (0.20,4.08) node[anchor=north west] {$Z_1$};
\draw (1.70,3.50) node[anchor=north west] {$-2$};
\draw (5.26,4.32) node[anchor=north west] {$Z_2$};
\draw (4.26,4.22) node[anchor=north west] {$-2$};
\draw (6.32,1.76) node[anchor=north west] {W};
\draw (5.32,1.96) node[anchor=north west] {$-3$};
\end{tikzpicture}
\caption{Configuration of the curves $Y_i, \, Z_j,\, W$ in Example \ref{ex:pen-6}}
\label{fig:M1}
\end{figure}

By applying Serrano's canonical bundle formula (\cite[Theorem 4.1]{Se96}) we obtain
\begin{equation*}
K_S = 2Y_1+Z_1 + 2Y_2 + Z_2  + W + Z_1 + Z_2 = F_1 +(2Y_2 - Y_1 + Z_2).
\end{equation*}
\end{example}
If $r(f_1)=1$ then $K_S - F_1$ would be numerically equivalent to an effective divisor (Proposition \ref{prop:car r=1}), hence
$2Y_2 - Y_1 +Z_2$ would be numerically equivalent to an effective divisor. On the other hand, we have
\begin{equation*}
(2Y_2 -Y_1 +Z_2)F_2 = -Y_1F_2 = -2 <0,
\end{equation*}
and this is a contradiction because $F_2$ is nef. It follows $r(f_1)=2$. The proof  of $r(f_2)=2$ is completely similar.

\subsection{Non-isotrivial examples}

\begin{example} \label{ex:K^2=6-d=2}
We give examples of non-isotrivial, irrational fibrations on surfaces with $p_g=q=2$, $K_S^2=6$
and Albanese map of degree $2$. Such surfaces were constructed and
classified in \cite{PP13b}. Here we only state the main result,
referring the reader to that paper for more details. Let
$(A, \,\mathcal{L})$ be a $(1, \,2)$-polarized abelian surface and
 denote by $\phi_2 \colon A[2] \to \widehat{A}[2]$ the
restriction of the canonical homomorphism
$\phi_{\mathcal{L}} \colon A \to \widehat{A}$ to the subgroup of
$2$-division points. Then $\ker \phi_2 =K(\mL) \cong (\mathbb{Z}/2 \mathbb{Z})^2$ and $\textrm{im}\, \phi_2$ consists of four
line bundles $\{\mathcal{O}_A, \, \mathcal{Q}_1, \, \mathcal{Q}_2,
\, \mathcal{Q}_3\}$; the set $\{\mathcal{Q}_1, \, \mathcal{Q}_2,
\, \mathcal{Q}_3\}$ will be denoted by $\textrm{im}\, \phi_2^{\times}$. The surfaces we are interested in can be constructed by using the following

\begin{proposition} \label{theorem:PP13-K^26}
Given an  abelian surface $A$ with a symmetric polarization $\mL$
of type $(1, \,2)$, not of product type,  for any $\mQ \in
\emph{im}\, \phi_2$ there exists a curve $D \in |\mL^2 \otimes
\mQ|$ whose unique non-negligible singularity is an ordinary
quadruple point at the origin $0 \in A$. Let $\mQ^{1/2}$ be a
square root of $\mQ$, and if $\mQ=\mO_A$ assume moreover
$\mQ^{1/2} \neq \mO_A$. Then the minimal desingularization $S$ of
the double cover of $A$ branched over $D$ and defined by $\mL
\otimes \mQ^{1/2}$ is a minimal surface of general type with
$p_g=q=2$, $K_S^2=6$ and whose Albanese map $\alpha \colon S \to A$ is a generically finite double cover.

Conversely, every minimal surface of general type with $p_g=q=2$,
$K_S^2=6$ and Albanese map of degree $2$ can be constructed in
this way.

Finally, the moduli space of these surfaces is a disjoint union of three connected, irreducible components
\begin{equation*}
\mathcal{M}=\mathcal{M}_{Ia} \sqcup \mathcal{M}_{Ib} \sqcup \mathcal{M}_{II}
\end{equation*}
of dimension $4, \, 4, \, 3$, respectively, where
\begin{itemize}
\item surfaces of type $Ia$ satisfy $\mQ= \mO_A$ and $\mQ^{1/2} \notin \rm{im}\, \phi_2^{\times};$
\item surfaces of type $Ib$ satisfy $\mQ= \mO_A$ and $\mQ^{1/2} \in \rm{im}\, \phi_2^{\times};$
\item surfaces of type $II$ satisfy $\mQ \in \rm{im} \, \phi_2^{\times}$.
\end{itemize}
\end{proposition}
\begin{proof}
See \cite[Theorems 2.6 and 3.7]{PP13b}.
\end{proof}
Now assume that the polarization  $\mL=\mO_A(L)$ on $A$ is of \emph{special type}, \
i.e. the linear system $|L|$ contains a member of the form $E'_1+E'_2$, where the $E'_i$
are two elliptic curves such that $E_1'E_2'=2$, see \cite[p. 46]{Ba87} or \cite[Section 1]{PP13a}. In particular $A$ is not simple, see Remark \ref{prop:A-not-simple}. The branch locus $D$ of $\alpha \colon S \to A$ intersect each $E_i'$ in four generically distinct points, in fact $D(E_1'+E_2')=2L^2=8$. An explicit construction of such a pair $(A,  \,\mL)$
is given in the Appendix, by taking a degree $2$ isogeny $\psi \colon A \to B$, where $B = E_1 \times E_2$ is the product of two elliptic curves. We now have two exact sequences of complex tori
\begin{equation*}
0 \to E_i' \to A \stackrel{\pi_i}{\to} E_i \to 0, \quad i=1,\, 2
\end{equation*}\
and the composition $\pi_i \circ \alpha$  gives
an irrational fibration $f_i \colon S \to E_i$, whose fibre $F_i$ has genus $3$. Hence the surface $S$ admits two irrational fibrations $f_1$ and $f_2$, both of genus $3$,
whose fibres $F_1$ and $F_2$ satisfy $F_1 F_2 = 2E_1'E_2' =4$.
By \cite[Remark 2.12]{PP13b} the general surface $S$ constructed in such a way has ample canonical class, hence by Proposition \ref{thm:numerical} it follows that $f_i \colon S \to E_i$ is not isotrivial.

We want now to compute $r(f_i)$. Let $\sigma \colon \widetilde{A} \to A$ be the
blow-up at the point $0 \in A$; we have a commutative diagram
\begin{equation*} \label{dia.alpha}
\xymatrix{
S \ar[r]^{\tilde{\alpha}} \ar[dr]_{\alpha} & \widetilde{A} \ar[d]^{\sigma} \\
 & A,}
\end{equation*}
where $\tilde{\alpha} \colon S \to \widetilde{A}$ is a flat double cover. Denote by $\Lambda \subset \widetilde{A}$ the exceptional divisor of $\sigma$; then the preimage of $\Lambda$ in $S$ is an elliptic curve $Z$ such that $Z^2 = -2$. The branch locus of $\tilde{\alpha}$ is a smooth curve $\widetilde{D} \in | \sigma^*(2L+Q)-4 \Lambda|$ and the square root of $\widetilde{D}$ determining the double cover is $\widetilde{L}:=\sigma^*(L+ Q^{1/2})-2 \Lambda$.
Hence we have
\begin{equation}  \label{eq:alpha-prime}
\begin{split}
\tilde{\alpha}_* \omega_S & = \omega_{\widetilde{A}} \oplus (\omega_{\widetilde{A}} \otimes \mO_{\widetilde{A}}(\widetilde{L})) \\
& = \mO_{\widetilde{A}}(\Lambda) \oplus \mO_{\widetilde{A}}(\sigma^*(L+ Q^{1/2})- \Lambda).
\end{split}
\end{equation}
The smooth elliptic fibration $\pi_i \colon A \to E_i$ induces, by composition with $\sigma \colon \widetilde{A} \to A$, an elliptic fibration $\tilde{\pi}_i \colon \widetilde{A} \to E_i$ with a unique singular fibre (the one containing the exceptional divisor $\Lambda$).  We define $\Gamma_{ip}$ to be the fibre of $\tilde{\pi}_i$ over $p \in E_i$.

Clearly $f_i = \tilde{\pi}_i \circ \tilde{\alpha}$, so by using \eqref{eq:alpha-prime} we can write
\begin{equation} \label{eq:fi*}
\begin{split}
f_{i*} \omega_{S/E_i} & = f_{i*} \omega_S = \tilde{\pi}_{i*} \tilde{\alpha}_* \omega_S \\
& = \tilde{\pi}_{i*} \mO_{\widetilde{A}}(\Lambda) \oplus \tilde{\pi}_{i*} \mO_{\widetilde{A}}(\sigma^*(L+  Q^{1/2})- \Lambda)) \\
& = \mO_{E_i} \oplus \tilde{\pi}_{i*} \mO_{\widetilde{A}}(\sigma^*(L+  Q^{1/2})- \Lambda)).
\end{split}
\end{equation}

\begin{proposition}\label{prop:rfinoniso}
We have $r(f_i)=1$ if and only if there exists a point $p \in E_i$ such that $H^0(\widetilde{A}, \, \sigma^*(L+Q^{1/2})-\Lambda-\Gamma_{ip}) \neq 0$. Otherwise $r(f_i)=2$.
\end{proposition}

\begin{proof}
Since $g(F_i)=3$ we have either $r(f_i)=1$ or $r(f_i)=2$. Looking at \eqref{eq:fi*} and arguing as in the proof of Proposition \ref{prop:car r=1}, we see that $r(f_i)=1$ if and only if there exists  $p \in E_i$ such that
$H^0(E_i, \,  \tilde{\pi}_{i*} \mO_{\widetilde{A}}(\sigma^*(L+  Q^{1/2})- \Lambda)\otimes \mO_{E_i}(-p)) \neq 0$. By projection formula this vector space has the same dimension as $ H^0(\widetilde{A}, \, \sigma^*(L+Q^{1/2})-\Lambda-\Gamma_{ip})$, so we are done.
\end{proof}
The pencil $|\mL \otimes \mQ^{1/2}|$ contains in general exactly two reducible divisors (\cite[Remark 1.16]{PP13a}). On the other hand, the point $0$ is not in the base locus of $|\mL \otimes \mQ^{1/2}|$ (because $\mathcal{Q}^{1/2} \neq \mathcal{O}_A$), so there is at most one reducible divisor in the pencil containing $0$.
Corollary \ref{cor:rfinoniso} below shows how the existence of such a divisor determines the integers $r(f_i)$.

\begin{corollary}\label{cor:rfinoniso}
The set $\{r(f_1), \, r(f_2)\}$ is as follows.
\begin{itemize}
\item[$\boldsymbol{(1)}$] if no reducible curve in $|\mL \otimes \mQ^{1/2}|$ contains $0$, then $\{r(f_1), \, r(f_2)\}= \{2\};$
\item[$\boldsymbol{(2)}$] if there exists a reducible curve in $|\mL \otimes \mQ^{1/2}|$ having a node at $0$, then $\{r(f_1), \, r(f_2)\}= \{1\};$
\item[$\boldsymbol{(3)}$] if there exists a reducible curve in $|\mL \otimes \mQ^{1/2}|$ which is smooth at $0$, then $\{r(f_1), \, r(f_2)\}= \{1, \, 2\}.$
 \end{itemize}
\end{corollary}
\begin{proof}
If no reducible curve in $|\mL \otimes \mQ^{1/2}|$ contains $0$, then $ H^0(\widetilde{A}, \, \sigma^*(L+Q^{1/2})-\Lambda-\Gamma_{ip})=0$ for $i=1, \, 2$, so Proposition \ref{prop:rfinoniso}  shows that $r(f_1)=r(f_2)=2$. This is case $\boldsymbol{(1)}$.

Therefore we can assume that there is a reducible curve $C \in |\mL \otimes \mQ^{1/2}|$ containing $0$. Using a slight abuse of notation we write $C=E_1' + E_2'$ (actually, $C$ is a translate of $E_1'+E_2' \in |\mL|$). Let $\widetilde{E}_i'$ be the strict transform of $E_i'$ via $\sigma \colon \widetilde{A} \to A$. There are two possibilities.
\begin{itemize}
\item $0$ is an ordinary double point for $C$. Then
\begin{equation*}
\sigma^*C - \Lambda = \sigma^* E_1' + \sigma^* E_2 ' - \Lambda = \widetilde{E}_1' + \widetilde{E}_2' + \Lambda.
\end{equation*}
Since $\widetilde{E}_i' + \Lambda$ is the (unique) reducible fibre of $\tilde{\pi}_i \colon \widetilde{A} \to E_i$, $i=1, \, 2$, by Proposition \ref{prop:rfinoniso} it follows $r(f_1)=r(f_2)=1$. This is case $\boldsymbol{(2)}$.
\item $0$ is a smooth point for $C$. Without loss of generality, we can assume that $0$ belongs to $E_1'$ but not to $E_2'$. Then
\begin{equation*}
\sigma^*C - \Lambda = \sigma^* E_1' + \sigma^* E_2 ' - \Lambda = \widetilde{E}_1' + \sigma^*E_2'.
\end{equation*}
Since $\sigma^*E_2'$ is a fibre of $\tilde{\pi}_2$ but $\widetilde{E}_1'$ is not a fibre of $\tilde{\pi}_1$, it follows $r(f_2)=1$, $r(f_1)=2$. This is case $\boldsymbol{(3)}$.
\end{itemize}
The proof is now complete.
\end{proof}
All the three cases in Corollary \ref{cor:rfinoniso} actually occur. We give explicit examples in the Appendix, see in particular  Proposition \ref{prop:node}, Proposition  \ref{prop:smooth} and Remark \ref{rem:types}.
\end{example}

\begin{example} \label{ex:K^2=5,6}
The same construction used in Example \ref{ex:K^2=6-d=2} can be used in other situations.
We describe a couple of examples, leaving the details to the reader.
\begin{itemize}
\item In \cite{PP13a} we studied some surfaces $S$ (originally constructed in \cite{CH06}) with $p_g=q=2$ and $K^2=5$. Their Albanese map $\alpha \colon S \to A$  is a generically finite triple cover of a $(1, \, 2)$-polarized abelian surface $(A, \, \mathcal{L})$, branched over a divisor $D \in |2 \mL|$ with an ordinary quadruple point. We can choose $A$ such that there is a degree $2$ isogeny $\psi \colon  A \to B$, where $B=E_1 \times E_2$ is the product of two elliptic curves. Then $S$ admits two non isotrivial, irrational fibrations $f_i \colon S \to E_i$, both with the general fibre of genus $3$.
\item In \cite{PP14} we constructed some new surfaces $S$ with $p_g=q=2$ and $K^2=6$. Their Albanese map $\alpha \colon S \to A$  is a generically finite quadruple cover of a $(1, \, 3)$-polarized abelian surface $(A, \, \mathcal{L})$, branched over a divisor $D \in |2 \mL|$ with
 six ordinary cusps. We can choose $A$ such that there is a degree $3$ isogeny $\psi \colon A \to B$, where $B=E_1 \times E_2$ is the product of two elliptic curves. Then $S$ admits two non isotrivial, irrational  fibrations $f_i \colon S \to E_i$, both with the general fibre of genus $4$.
\end{itemize}
\end{example}

\begin{example} \label{ex:diagonal}
We can further specialize the construction described in Examples \ref{ex:K^2=6-d=2} and \ref{ex:K^2=5,6}, assuming that $E_1=E_2$, so that the abelian surface $A$ is isogenous to the product $B=E \times E$ of an elliptic curve $E$ with itself. 

This allows us to obtain examples with \emph{infinitely many} irrational fibrations whose fibre genera are arbitrarily large. 
In fact, for any $n \geq 1$ let us consider the elliptic fibration $g_n \colon B \to E$ defined by $(x, \, y) \mapsto x \oplus ny$, where $\oplus$ is the group law on $E$. Composing with a degree $2$ (resp. a degree $3$) isogeny $\psi \colon A \to B$, 
we obtain an elliptic fibration $h_n \colon A \to E$. Moreover, we can choose $\psi$ in such a way that the induced $(1, \,2)$-polarization (resp. $(1, \, 3)$-polarization) $\mL=\mO_A(L)$ on $A$ is not of product type: see the Appendix, where the case $\deg \psi=2$ is discussed in detail.

If $E_n$ is the general fibre of $h_n$, we have $\lim_{n \rightarrow + \infty} E_nL = + \infty.$ Therefore, repeating the previous constructions, we can build families of surfaces $S$  with $p_g=q=2$ and $(K_S^2, \, \deg \alpha) \in \{(5,\, 3), \, (6, \, 2), \, (6, \, 4) \}$, such that $S$ contains an infinite family $f_n \colon S \to E$ of non-isotrivial, irrational fibrations. Moreover the fibre of $f_n$ has genus strictly increasing with $n$, hence Corollary \ref{prop:genus fibre} implies $r(f_n) \geq 2$ for almost all $n$.

These examples demonstrate that the (still incomplete) classification of irrational fibrations on surfaces with $p_g=q=2$ and maximal Albanese dimension needs to be much subtler than the one attempted in \cite{Z03}.  
\end{example}

\begin{example} \label{ex:r non limited}
The same idea of Example \ref{ex:diagonal} allows us to produce examples with $(K_S^2, \, \deg \alpha)=(4, \, 2)$ and $r(f)$ arbitrarily large. In fact, let $B=E \times E$ and let $\mathscr{L} = \mO_B(L)$ be a principal product polarization on $B$. Take the fibration $g_n \colon B \to E$ defined as above and let $E_n$ be its fibre; then $E_nL = n^2+1$. The linear system $|2L|$ is base-point free, see \cite[Lemma 2.8]{Ke91}, so the general curve $D \in |2L|$ is smooth. Therefore the double cover $S \to B$  branched over $D$ is a minimal surface of general type with $p_g=q=2$ and $K_S^2=4$. Moreover the infinite family of fibrations $g_n$ induce a family of irrational fibrations $f_n \colon S \to E$, whose fibre $F_n$ has genus $g(F_n)=n^2+2$.    
Hence all the fibrations $f_n$ are distinct and by Propositions \ref{prop:xiao} and \ref{thm:main} we have 
\begin{equation*}  
(f_n)_* \omega_{S/E} = \mO_E \oplus \mE_p(n^2+1, \, 1),
\end{equation*}
which gives $r(f_n)=n^2+1$.
\end{example}

\begin{remark} \label{rm:interesting}
It would be interesting to compute the integer $r(f)$ for all the surfaces in Examples \ref{ex:K^2=5,6} and \ref{ex:diagonal}.
\end{remark}

\begin{remark} \label{rm: genus2}
The situation described in Examples \ref{ex:diagonal} and \ref{ex:r non limited} can only occur for irrational pencils over an elliptic base. In fact, a classical result of Severi states that a surface of general type has at most finitely many pencils over curves of genus $\geq 2$, see \cite[Section 2]{MLP11}. 
\end{remark}


\section*{Appendix: explicit computations on abelian surfaces with $(1, \, 2)$-polarization of special type}

We start with a principally polarized abelian surface $B$ which is the product of two elliptic curves, i.e.  $B:=E_1 \times E_2$. Then the period matrix of $B$ is

\begin{equation*}
\left(
\begin{array}{cccc}
\tau_1 & 0 & 1 & 0 \\
0 & \tau_2 & 0 & 1 \\
\end{array}
\right),
\end{equation*}
where ${\rm Im}(\tau_i) >0$ for $i=1, \, 2$; hence $B= \mathbb{C}^2/\Lambda_B$, the lattice $\Lambda_B$ being spanned by the four column vectors
\begin{equation*}
\lambda_1:=
\left(
\begin{array}{c}
\tau_1 \\
0
\end{array}
\right), \quad
\lambda_2:=
\left(
\begin{array}{c}
0 \\
\tau_2
\end{array}
\right), \quad
\mu_1:=
\left(
\begin{array}{c}
1 \\
0
\end{array}
\right), \quad
\mu_2:=
\left(
\begin{array}{c}
0 \\
1
\end{array}
\right).
\end{equation*}
 Notice that
\begin{equation*}
E_1= \mathbb{C}/(\tau_1 \mathbb{Z} \oplus \mathbb{Z}), \quad E_2= \mathbb{C}/(\tau_2 \mathbb{Z} \oplus \mathbb{Z}).
\end{equation*}
There is a natural principal polarization of product type on $B$, which is induced by the alternating form $E_B \colon \Lambda_B \times \Lambda_B \to \mathbb{Z}$, where
\begin{equation*}
E_B(\lambda_1, \, \mu_1)=1, \quad E_B(\mu_1, \, \lambda_1)=-1, \quad E_B(\lambda_2, \, \mu_2)=1, \quad E_B(\mu_2, \, \lambda_2)=-1
\end{equation*}
and all the other values are zero. Let $\widehat{B}:=\textrm{Pic}^0(B)$ be the dual abelian variety of $B$. By the Appell-Humbert Theorem its elements can be identified with the characters $\Lambda_B \to \mathbb{C}^*$; we will indicate such a character $\chi^B$ by the vector
\begin{equation*}
(\chi^B(\lambda_1), \, \chi^B(\lambda_2), \,\chi^B(\mu_1), \,\chi^B(\mu_2)).
\end{equation*}
The principal polarization yields an isomorphism $B \to \widehat{B}$, sending the point $x \in B$ to the character ${\rm exp}(2\pi i E_B(\cdot , \, x))$.

The finite subgroup $\widehat{B}[2]$ of  $\widehat{B}$ is isomorphic to $(\mathbb{Z}/2\mathbb{Z})^4$ and corresponds to the $16$ characters $\Lambda_B \to \mathbb{C}^*$ with values in $\{\pm 1\}$.
The four characters giving the $2$-torsion line bundles which are pullback from $E_1$  are $\textrm{exp}(2 \pi i E_B( \cdot, \, x))$, where $x$ is one of the points $0, \, \frac{\lambda_1}{2}, \,  \frac{\mu_1}{2}, \, \frac{\lambda_1+\mu_1}{2}$, so they can be written as
\begin{equation} \label{eq:char-E1}
\begin{split}
\chi^B_0 & = (1, \, 1,\, 1,\, 1), \quad \chi^B_1 = (1, \, 1,\, -1,\, 1), \\
\chi^B_2 & = (-1, \, 1,\, 1,\, 1), \quad \chi^B_3 = (-1, \, 1,\, -1,\, 1).
\end{split}
\end{equation}
Analogously, the four characters giving the $2$-torsion line bundles which are pullback from $E_2$  are $\textrm{exp}(2 \pi i E_B( \cdot, \, x))$, where $x$ is one of the points $0, \, \frac{\lambda_2}{2}, \,  \frac{\mu_2}{2}, \, \frac{\lambda_2+\mu_2}{2}$, so they can be written as
\begin{equation} \label{eq:char-E2}
\begin{split}
\chi^B_0 & = (1, \, 1,\, 1,\, 1), \quad \chi^B_4 = (1, \, 1,\, 1,\, -1), \\ \chi^B_5 & = (1, \, -1,\, 1,\, 1), \quad \chi^B_6 = (1, \, -1,\, 1,\, -1).
\end{split}
\end{equation}
Multiplying the four characters in \eqref{eq:char-E1} with those in \eqref{eq:char-E2} we obtain all the $16$ characters in $\widehat{B}[2]$.

Now we construct an abelian surface $A$ with a  symmetric $(1, \, 2)$-polarization of special type (and which is not a product polarization) as a double cover of $B$. We start by a sublattice
$\Lambda_A \subset \Lambda_B$ of index $2$, and then we take $A:=\mathbb{C}^2/\Lambda_A$.
For instance we may consider
\begin{equation*}
\Lambda_A := \lambda_1 \mathbb{Z} \oplus (\lambda_1+\lambda_2) \mathbb{Z} \oplus (\mu_1-\mu_2) \mathbb{Z} \oplus 2 \mu_2 \mathbb{Z}.
\end{equation*}
The alternating form $E_B$ induces an alternating form $E_A \colon \Lambda_A \times \Lambda_A \to \mathbb{Z}$ given by
\begin{equation*}
E_A(\lambda_1, \, \mu_1-\mu_2)=1, \quad E_A( \mu_1-\mu_2, \, \lambda_1)=-1, \quad E_A(\lambda_1+ \lambda_2, \, 2\mu_2)=2, \quad E_A(2\mu_2, \, \lambda_1+\lambda_2)=-2
\end{equation*}
and all the other values are zero. This defines a symmetric $(1, \, 2)$-polarization $\mL = \mO_A(L)$ on $A$, such that
\begin{equation*}
K(\mL) = \big \langle \mu_2, \, \frac{\lambda_1+\lambda_2}{2} \big \rangle.
\end{equation*}
We indicate a character $\chi^A \colon \Lambda_A \to \mathbb{C}^*$ by the vector
\begin{equation*}
(\chi^A(\lambda_1), \, \chi^A(\lambda_1 + \lambda _2 ), \,\chi^A(\mu_1 - \mu_2), \,\chi^A(2\mu_2)).
\end{equation*}
The degree $2$ isogeny $\psi \colon A \to B$ induces a degree $2$ isogeny $\widehat{\psi} \colon \widehat{B}  \to \widehat{A}$, obtained by restriction of the characters $\chi^B \colon \Lambda_B \to \mathbb{C}^*$ to the sublattice $\Lambda_A$. In other words, $\widehat{\psi}(\chi^B)=\chi^A$, where $\chi^A$ is defined by
\begin{equation} \label{eq:chi}
\begin{split}
\chi^A(\lambda_1) & = \chi^B(\lambda_1), \\
\chi^A(\lambda_1+\lambda_2) & = \chi^B(\lambda_1+\lambda_2) = \chi^B(\lambda_1)\chi^B(\lambda_2), \\
\chi^A(\mu_1-\mu_2) & = \chi^B(\mu_1-\mu_2)=\chi^B(\mu_1)\chi^B(\mu_2)^{-1},\\
\chi^A(2\mu_2) & = \chi^B(2\mu_2)=\chi^B(\mu_2)^{2}.\\
\end{split}
\end{equation}
By \eqref{eq:chi} it follows immediately that $\ker \widehat{\psi}$ is the group of order $2$ generated by $\chi_1^B \chi_4^B = (1, \, 1, \, -1, \, -1)$.
Notice that $\chi^B_1\chi^B_4$ is the character ${\rm exp}(2\pi i E_B(\cdot, \, \frac{\lambda_1+ \lambda_2}{2}))$, that is the image of $\frac{\lambda_1 +\lambda_2}{2}$ via the isomorphism $B \to \widehat{B}$. On the other hand, the generator of $\ker \widehat{\psi}$ corresponds to the $2$-torsion line bundle on $B$ inducing the \'{e}tale double cover $A \to B$; since it is not a pullback of a line bundle from $E_1$ or from $E_2$, it follows that the $(1,\, 2)$-polarization $\mL$ on $A$ is not of product type.

We can also give the following interpretation of the $16$ characters $\Lambda_A \to \{\pm 1\}$ corresponding to the $2$-torsion line bundles on $A$: eight of them arise from the $2$-torsion line bundles on $B$, namely
\begin{equation*}
\begin{array}{ll}
\chi^A_0  = (1, \,1, \, 1, \, 1), &   \chi^A_1 = (1, \,1, \, -1, \, 1) \\
\chi^A_2  = (-1, \,-1, \, 1, \, 1), &  \chi^A_3 = (-1, \,-1, \, -1, \, 1), \\
\chi^A_5  = (1, \,-1, \, 1, \, 1), & \chi^A_1 \chi^A_5 = (1, \,-1, \, -1, \, 1), \\
\chi^A_2 \chi^A_5  = (-1, \,1, \, 1, \, 1), & \chi^A_3\chi^A_5 = (-1, \,1, \, -1, \, 1),
\end{array}
\end{equation*}
whereas the remaining eight are the images in $\widehat{A}$ of the $16$ square roots of  the generator of $\ker \widehat{\psi}$:
\begin{equation*}
\begin{array}{ll}
\varepsilon_1 = (1, \,1, \, 1, \, -1),   &   \varepsilon_2 = (1, \,1, \, -1, \, -1) \\
\varepsilon_3  = (1, \,-1, \, 1, \, -1), & \varepsilon_4 = (1, \,-1, \, -1, \, -1), \\
\varepsilon_5  = (-1, \,-1, \, 1, \, -1), &  \varepsilon_6 = (-1, \,-1, \, -1, \, -1), \\
\varepsilon_7  = (-1, \,1, \, 1, \, -1),& \varepsilon_8 = (-1, \,1, \, -1, \, -1).
\end{array}
\end{equation*}

Consider now the isogeny induced by $\mL$, namely
\begin{equation*}
\phi = \phi_{\mL} \colon A \to \widehat{A}, \quad x \mapsto t^*_x \mL \otimes \mL^{-1}.
\end{equation*}
Then the restricted map $\phi_2 \colon A[2] \to \widehat{A}[2]$ satisfies
$\ker \phi_2=K(\mL)$.
The image of $\phi_2$ is generated by $\phi_2(\frac{\lambda_1}{2})$ and $\phi_2(\frac{\mu_1-\mu_2}{2})$. Since the line bundle $\phi(x)$ corresponds to the character ${\rm exp}(2\pi i E_A(\cdot, x))$, by simple computations we obtain
\begin{equation*}
{\rm im}\,  \phi_2 = \{\chi^A_0, \,  \chi^A_1, \, \chi^A_2 \chi^A_5, \, \chi^A_3 \chi^A_5 \}.
\end{equation*}

The complete linear system $|\mathscr{L}|$ contains exactly two reducible divisors $E_1'+E_2'$ and ${E_1'}^*+{E_2'}^*$, which are the pullback via $\psi$ of two reducible divisors $E_1+E_2$ and $E_1^*+E_2^*$ on $B$; moreover ${E_1'}^*+{E_2'}^*$ is the translated of
$E_1'+E_2'$ by the point $\frac{\lambda_1+\lambda_2}{2}$, see Figure \ref{fig:M2}.

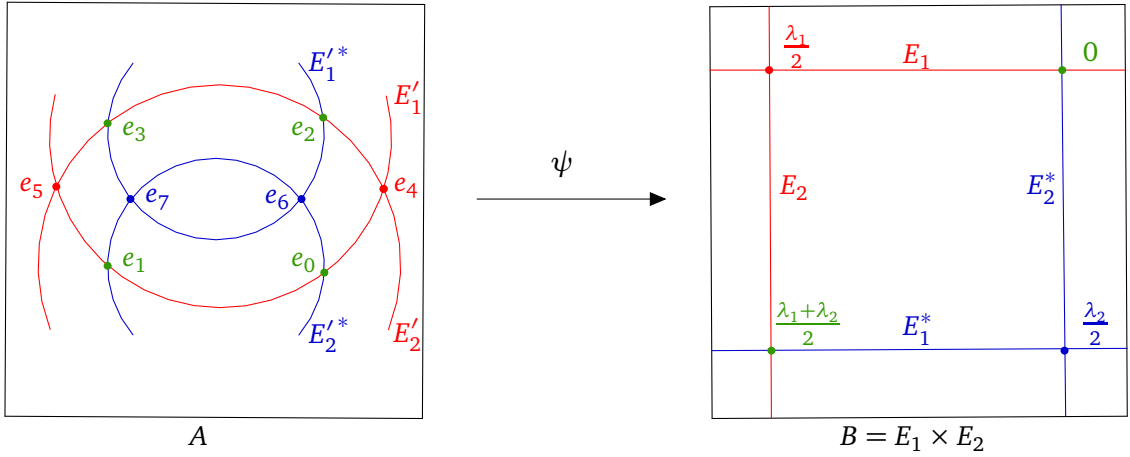
\definecolor{ttffqq}{rgb}{0.2,0.6.5,0}
\definecolor{qqqqcc}{rgb}{0,0,0.8}
\definecolor{ffqqqq}{rgb}{1,0,0}
\definecolor{xdxdff}{rgb}{0.49,0.49,1}
\definecolor{uququq}{rgb}{0.25,0.25,0.25}
\definecolor{zzttqq}{rgb}{0.6,0.2,0}
\begin{figure}[H]
\centering
\begin{tikzpicture}[line cap=round,line join=round,>=triangle 45,x=1.5cm,y=1.5cm]
\clip(-4.2,-4.2) rectangle (25,2);
\draw  (-0.46,1.72)-- (-4.12,1.74);
\draw  (-4.12,1.74)-- (-4.14,-1.92);
\draw  (-4.14,-1.92)-- (-0.48,-1.94);
\draw  (-0.48,-1.94)-- (-0.46,1.72);
\draw  (2.06,-1.94)-- (5.7,-1.92);
\draw (5.7,-1.92)-- (5.68,1.72);
\draw  (5.68,1.72)-- (2.04,1.7);
\draw  (2.04,1.7)-- (2.06,-1.94);
\draw [->] (0,0) -- (1.66,0);
\draw [shift={(-2.25,0.54)},color=ffqqqq]  plot[domain=-3.4:0.25,variable=\t]({1*1.5*cos(\t r)+0*1.5*sin(\t r)},{0*1.5*cos(\t r)+1*1.5*sin(\t r)});
\draw [shift={(-2.26,-0.58)},color=ffqqqq]  plot[domain=-0.37:3.51,variable=\t]({1*1.59*cos(\t r)+0*1.59*sin(\t r)},{0*1.59*cos(\t r)+1*1.59*sin(\t r)});
\draw [shift={(-2.29,0.59)},color=qqqqcc]  plot[domain=-3.84:0.7,variable=\t]({1*0.95*cos(\t r)+0*0.95*sin(\t r)},{0*0.95*cos(\t r)+1*0.95*sin(\t r)});
\draw [shift={(-2.29,-0.59)},color=qqqqcc]  plot[domain=-0.7:3.84,variable=\t]({1*0.95*cos(\t r)+0*0.95*sin(\t r)},{0*0.95*cos(\t r)+1*0.95*sin(\t r)});
\draw [color=ffqqqq] (2.04,1.14)-- (5.68,1.14);
\draw [color=ffqqqq] (2.56,1.7)-- (2.58,-1.94);
\draw [color=qqqqcc] (2.06,-1.34)-- (5.7,-1.32);
\draw [color=qqqqcc] (5.12,1.72)-- (5.16,-1.92);
\draw (0.56,0.52) node[anchor=north west] {$\psi$};
\draw (-2.6,-1.92) node[anchor=north west] {$A$};
\draw (3.1,-1.92) node[anchor=north west] {$B=E_1 \times E_2$};
\fill [color=ffqqqq] (-0.82,0.09) circle (1.5pt);
\draw[color=ffqqqq] (-3.90,0.10) node {$e_5$};
\fill [color=ffqqqq] (-3.69,0.11) circle (1.5pt);
\fill [color=ffqqqq] (-0.82,0.09) circle (1.5pt);
\draw[color=ffqqqq] (-0.62,0.10) node {$e_4$};
\draw[color=ffqqqq] (-0.60,0.9) node {$E'_1$};
\draw[color=ffqqqq] (-0.61,-1.2) node {$E'_2$};
\draw[color=qqqqcc] (-1.30,1.2) node {${E'_1}^*$};
\draw[color=qqqqcc] (-1.30, -1.2) node {${E'_2}^*$};
\fill [color=qqqqcc] (-1.54,0) circle (1.5pt);
\draw[color=qqqqcc] (-1.75,0) node {$e_6$};
\fill [color=qqqqcc] (-3.04,0) circle (1.5pt);
\draw[color=qqqqcc] (-2.80,0) node {$e_7$};
\fill [color=ttffqq] (-1.34,-0.65) circle (1.5pt);
\draw[color=ttffqq] (-1.52,-0.56) node {$e_0$};
\fill [color=ttffqq] (-3.24,-0.59) circle (1.5pt);
\draw[color=ttffqq] (-3.00,-0.56) node {$e_1$};
\fill [color=ttffqq] (-1.35,0.72) circle (1.5pt);
\draw[color=ttffqq] (-1.52,0.60) node {$e_2$};
\fill [color=ttffqq] (-3.24,0.67) circle (1.5pt);
\draw[color=ttffqq] (-3.00,0.60) node {$e_3$};
\fill [color=ffqqqq] (2.56,1.14) circle (1.5pt);
\draw[color=ffqqqq] (2.80,1.35) node {$\frac{\lambda_1}{2}$};
\draw[color=ffqqqq] (3.86,1.26) node {$E_1$};
\draw[color=ffqqqq] (2.78,0.10) node {$E_2$};
\draw[color=qqqqcc] (3.86,-1.15) node {$E_1^*$};
\draw[color=qqqqcc] (4.95,0.10) node {$E_2^*$};
\fill [color=qqqqcc] (5.15,-1.34) circle (1.5pt);
\draw[color=qqqqcc] (5.42,-1.10) node {$\frac{\lambda_2}{2}$};
\fill [color=ttffqq] (5.13,1.14) circle (1.5pt);
\draw[color=ttffqq] (5.38,1.30) node {$0$};
\fill [color=ttffqq] (2.58,-1.34) circle (1.5pt);
\draw[color=ttffqq] (2.90,-1.10) node {$\frac{\lambda_1+\lambda_2}{2}$};
\draw [color=ttffqq](-2.22,-2.56) node[anchor=north west] {\parbox{8.86 cm}{$e_0=0, \; \; e_1=\mu_2,  \; \; e_2=\frac{\lambda_1+\lambda_2}{2}, \; \; e_3=\mu_2+\frac{\lambda_1+\lambda_2}{2},$}};
\draw [color=ffqqqq](-2.2,-3.12) node[anchor=north west] {\parbox{4.26 cm}{$e_4=\frac{\lambda_1}{2}, \; \; e_5= \mu_2+\frac{\lambda_1}{2},$}};
\draw [color=qqqqcc](-2.22,-3.76) node[anchor=north west] {\parbox{4.18 cm}{$e_6=\frac{\lambda_2}{2}, \; \; e_7=\mu_2+\frac{\lambda_2}{2}.$}};
\end{tikzpicture}
\caption{The degree $2$ isogeny $\psi\colon A \to B$.}
\label{fig:M2}
\end{figure}

\begin{lemma} \label{lemma:reducible}
Fix $x \in A$. Then the linear system $|t_x^*\mathscr{L}|$ contains a reducible element $D$ passing through $0$ if and only if we are in one of the following cases:
\begin{itemize}
\item[$\boldsymbol{(1)}$] $\psi(x)= (y_1, \, 0);$
\item[$\boldsymbol{(2)}$] $\psi(x)= (y_1, \, \frac{\tau_2}{2});$
\item[$\boldsymbol{(3)}$] $\psi(x)= (0, \, y_2);$
\item[$\boldsymbol{(4)}$] $\psi(x)= (\frac{\tau_1}{2}, \, y_2),$
\end{itemize}
where we use the notation $(y_1, \, y_2) \in B=E_1 \times E_2$.
\end{lemma}
\begin{proof}
We refer again to Figure \ref{fig:M2}. The curve $\psi(D) \subset B$ is a translate of $E_1+E_2$ containing $0$. This means that either $E_1$ or $E_2^*$ must be a component of $\psi(D)$. In the former case, either $E_1$ is fixed by the translation by $\psi(x)$ (case $\boldsymbol{(1)}$) or such a translation sends $E_1^*$ to $E_1$ (case $\boldsymbol{(2)}$). In the latter case, either $E_2^*$ is fixed by the translation by $\psi(x)$ (case $\boldsymbol{(3)}$) or such a translation sends $E_2$ to $E_2^*$ (case $\boldsymbol{(4)}$). This completes the proof.
\end{proof}

\begin{proposition}\label{prop:node}
There exists a reducible curve $D \in |\mL \otimes \mQ^{1/2}|$ such that $0$ is an ordinary double point for $D$ if and only if $\mathcal{Q}=\mathcal{O}_A$ and $\mathcal{Q}^{1/2}$ corresponds to the character $\chi^A_1$.
\end{proposition}
\begin{proof}
If $0$ is an ordinary double point for $D$, looking at Figure \ref{fig:M2} we see that $|\mL \otimes \mQ^{1/2}|$ has to be in the linear system containing the translate of $E'_1+E'_2$ by $\frac{\lambda_1}{2}$ (or, which is the same, the translate of ${E_1'}^*+{E_2'}^*$ by $\frac{\lambda_2}{2}$). In other words we must have
$\mathcal{L}\otimes \mathcal{Q}^{1/2}=t^*_{\lambda_1/2} \mL$.
By \cite[Lemma 2.3.2]{BL04} it follows that $\mathcal{Q}^{1/2}$ corresponds to the character ${\rm exp}(E_A(\cdot, \, \frac{\lambda_2}{2}))$, which is precisely $\chi^A_1$. Notice that, in the notation of Lemma \ref{lemma:reducible}, this situation corresponds to either case $\boldsymbol{(1)}$ with $\psi(x)=(\frac{\tau_1}{2}, \, 0)$ or to case $\boldsymbol{(3)}$ with $\psi(x)=(0, \, \frac{\tau_2}{2})$.
\end{proof}

\begin{proposition}\label{prop:smooth}
Take $\mathcal{Q} \in \emph{im} \,\phi_2$. Then there exists a reducible curve $D \in |\mL \otimes \mQ^{1/2}|$ such that $0$ is a smooth point of $D$ if and only if we are in one of the following two cases:
\begin{itemize}
\item[$\boldsymbol{(a)}$] $\mathcal{Q}=\mathcal{O}_A$ and $\mathcal{Q}^{1/2}$ corresponds to one of the four characters $\chi^A_2, \, \chi^A_3, \, \chi^A_5, \, \chi^A_1 \chi^A_5;$
\item[$\boldsymbol{(b)}$] $\mathcal{Q}$ corresponds to the character  $\chi^A_1 \in {\rm im}\, \phi^{\times}_2$.
\end{itemize}
\end{proposition}
\begin{proof}
If a curve as in the statement exists, by Lemma \ref{lemma:reducible} we see that $\mL \otimes  \mQ^{1/2} = t_x^* \mL$, where either $\psi(x)=(\frac{\tau_1}{2}, \, y_2)$ or $\psi(x)=(y_1, \, \frac{\tau_2}{2})$ and $y_i$ has either order $2$ or order $4$.
Now a tedious but straightforward computation shows that if $y_i$ has order $2$ then $x \in \{\frac{\mu_1}{2}, \, \frac{\lambda_1+\mu_1}{2}, \, \frac{\mu_2}{2}, \, \frac{\lambda_2+\mu_2}{2} \}$ and we are in case $\boldsymbol{(a)}$, whereas if $y_i$ has order $4$ we are in case $\boldsymbol{(b)}$.
\end{proof}

\begin{remark} \label{rem:types}
Using the terminology of \cite{PP13b} (explained in Proposition \ref{theorem:PP13-K^26}), the cases in Proposition \ref{prop:node} belong to  \emph{surfaces of type Ib}, those in Proposition \ref{prop:smooth}, $\boldsymbol{(a)}$  belong to \emph{surfaces of type Ia} and those in Proposition \ref{prop:smooth}, $\boldsymbol{(b)}$ to \emph{surfaces of type II}.
\end{remark}
%

\bigskip
\bigskip
\bigskip

Matteo Penegini \\Dipartimento di Matematica ``Federigo Enriques'', Universit\`{a} degli Studi di Milano, Via Saldini 50, 20133 Milano, Italy \\
\emph{E-mail address:}
 \verb|matteo.penegini@unimi.it| \\ \\

Francesco Polizzi \\ Dipartimento di Matematica e Informatica,
Universit\`{a} della Calabria, Cubo 30B, 87036 Arcavacata di Rende
(Cosenza), Italy \\ \emph{E-mail address:}
\verb|polizzi@mat.unical.it|

\end{document}